\documentclass[10pt]{amsart}
\usepackage{amsmath}
\usepackage{amssymb}
\usepackage{amsthm}
\usepackage{amsfonts}
\usepackage{setspace}
\usepackage{cite}
\usepackage[usenames,dvipsnames,svgnames,table]{xcolor}
\usepackage[colorlinks=true, citecolor=blue, linkcolor=Crimson]{hyperref}

\usepackage{color, etoolbox, xspace}
\newtoggle{draftmode}
\newcommand{\highlight}[2][red]{{\color{#1}#2}}
\newcommand{\bsc}{\usefont{T1}{cmr}{bx}{sc}}

\newcommand{\note}[1][]{\iftoggle{draftmode}{{\noindent\highlight{\bsc\ifblank{#1}{[Achtung]}{[{\rm #1}]}}}\xspace}{}}

\newcommand{\DNC}{\mathrm{DNC}}

\newcommand{\pair}[1]{\langle #1 \rangle}

\renewcommand{\deg}[1]{\mathbf{#1}}

\newcommand{\estr}{\langle\rangle}

\newcommand{\bstrings}{\omega^{<\omega}}
\newcommand{\dset}[2]{\{#1 : #2 \}}
\newcommand{\tH}{\textrm{th}}
\newcommand{\ML}{Martin-L\"{o}f}

\DeclareMathOperator{\res}{\upharpoonright}
\DeclareMathOperator{\upto}{\upharpoonright}
\DeclareMathOperator{\converges}{\downarrow}
\DeclareMathOperator{\diverges}{\uparrow}

\theoremstyle{plain}
\newtheorem{theorem}{Theorem}[section]

\newtheorem{corollary}[theorem]{Corollary}
\newtheorem{lemma}[theorem]{Lemma}
\newtheorem{claim}[theorem]{Claim}

\theoremstyle{definition}
\newtheorem{definition}[theorem]{Definition}
\newtheorem{question}[theorem]{Question}

\togglefalse{draftmode}

\newcommand{\pP}{\mathbb{P}}
\newcommand{\floor}[1]{\lfloor #1 \rfloor}

\title{Forcing with Bushy Trees}

\author{Mushfeq Khan}
\address[Mushfeq Khan]{Department of Mathematics\\
University of Hawai`i at M{\=a}noa\\
Honolulu, HI 96822, USA}
\email{khan@math.hawaii.edu}
\author{Joseph S.~Miller}
\address[Joseph S.~Miller]{Department of Mathematics\\
University of Wisconsin\\
Madison, WI 53706-1388, USA}
\email{jmiller@math.wisc.edu}

\begin{document}

\begin{abstract} We present several results that rely on arguments involving the combinatorics of ``bushy trees''. These include the fact that there are arbitrarily slow-growing diagonally noncomputable ($\DNC$) functions that compute no Kurtz random real, as well as an extension of a result of Kumabe in which we establish that there are $\DNC$ functions relative to arbitrary oracles that are of minimal Turing degree. Along the way, we survey some of the existing instances of bushy tree arguments in the literature. \end{abstract}

\maketitle

\section{Introduction}

	In 1985, Sacks~\cite{Sacks} asked if there exist diagonally noncomputable (or $\DNC$) functions of minimal Turing degree. Kumabe answered the question in 1993, constructing such a function and pioneering the application of bushy tree arguments in computability theory. A draft of the proof ~\cite{Kumabe} was in private circulation by 1996, but has remained unpublished. \note[provide more detail on this draft]

	Arguments involving bushy trees and their combinatorics have since been applied to several questions concerning $\DNC$ functions. In 2000, Simpson and Giusto \cite{GiustoSimpson} asked if the reverse mathematics axiom system $\mathrm{DNC}$ is stronger than the system $\mathrm{WWKL_0}$. In 2004, Ambos-Spies, Kjos-Hanssen, Lempp, and Slaman \cite{DNRWWKL} used ideas from Kumabe's proof to provide an affirmative answer.

	In 2009, motivated by questions around Yates's long-standing open problem about whether every minimal degree has a strong minimal cover, Lewis collaborated with Kumabe to produce a simplified version \cite{KumabeLewis} of Kumabe's proof, the publication of which introduced the technique of ``bushy tree forcing'' to the wider community.

	A simpler variation on the technique appeared in Greenberg and Miller's 2011 result \cite{MillerGreenberg} that there are arbitrarily slow-growing $\DNC$ functions that compute no \ML\ random real. 

	More recently, Beros \cite{Beros} has applied arguments involving bushy trees to show that there exist $\DNC$ functions that compute no effectively bi-immune set, answering a question of Jockusch and Lewis \cite{JockuschLewis}. Dorais, Hirst, and Shafer \cite{DoraisHirstShafer}, building on the aforementioned work of Ambos-Spies, et al. \cite{DNRWWKL}\note[double-check], have shown that the reverse mathematics principle ``there exists a $k$ such that for every function $f$ there is a $k$-bounded function that is $\DNC$ relative $f$'' does not imply the existence of a $\{0, 1\}$-valued $\DNC$ function in the absence of $\Sigma^0_2$ induction, answering a question of Simpson. Bienvenu and Patey \cite{Fireworks}, by combining bushy tree arguments with probabilistic ones, have shown that there is a computable function $h$ such that every $2$-random real computes an $h$-bounded $\DNC$ function that computes no \ML\ random real.

	Of the new results we present here, there are two main ones. Theorem~\ref{thm:slow-no-Kurtz} is a variation on the Greenberg-Miller result mentioned above, stronger in one aspect, but (necessarily) weaker in another: There are arbitrarily slow-growing $\DNC$ functions that compute no Kurtz random real, although this fact cannot be partially relativized to yield a $\DNC$ function relative to an arbitrary oracle. It is a consequence of this theorem that there are sequences of effective Hausdorff dimension 1 that compute no Kurtz random real. Theorem~\ref{thm:minimal-X-DNC}, due to the first author, is a partial relativization of Kumabe's theorem. It asserts the existence of $\DNC$ functions relative to arbitrary oracles that are of minimal Turing degree.  

	One of the goals of the current paper is to study the diverse applications of bushy tree arguments in computability theory with a view to understanding what the similarities and differences between them are. In the case of forcing arguments, we are particularly interested in how properties of the partial order determine properties of the generic object (typically a $\DNC$ function). The definitions and combinatorial lemmas in Section~\ref{sec:defs-and-combinatorics} underly all of the arguments we present, and encapsulate some of the similarities. 

	The differences can be seen to occur primarily along three ``axes''. The first of these relates to the nature of the approximation to the generic object. In some arguments, the approximations are finite strings (what we term ``basic bushy forcing''), while others involve maintaining infinite trees. A second major difference is in the complexity of what we label the ``bad set''. These are sets of strings that are declared to be off limits in a construction. Some arguments (Theorem~\ref{thm:slow-no-Kurtz}, for example) require that the bad sets be computably enumerable, and these are not automatically amenable to partial relativization. In others, dropping the assumption of any form of effectivity on the bad set allows partial relativization (as in Theorem~\ref{thm:minimal-X-DNC}), but may require more complicated combinatorics, or a different assumption on the effectivity of the approximation. The third major difference is in whether the resulting $\DNC$ function can be constructed pointwise below a preimposed order function. This is possible, for example, in Theorem~\ref{thm:slow-no-Kurtz}, but the question of whether it is possible in Theorem~\ref{thm:minimal-X-DNC} is an important open one.

\section{Definitions and combinatorial lemmas}
\label{sec:defs-and-combinatorics}
Let $\varphi_0$, $\varphi_1$, $\varphi_2$, ... be an effective enumeration of the partial computable functions. The partial computable function $e \mapsto \varphi_e(e)$ is called the \emph{diagonal partial computable} function.

\begin{definition}
	A function $f \in \omega^\omega$ is \emph{diagonally noncomputable}, or \emph{$\DNC$}, if for all $e$ such that $\varphi_e(e)$ converges, $f(e) \neq \varphi_e(e)$.
\end{definition}

Of particular interest to us are the $\DNC$ functions that are bounded by some computable function $h \in \omega^\omega$.

\begin{definition}
	Let $h \in \omega^\omega$ be computable and such that for all $n \in \omega$, $h(n) \ge 2$. Then $\DNC_h$ denotes the class of $\DNC$ functions $f$ such that for all $n \in \omega$, $f(n) < h(n)$. The class of $\DNC$ functions in $k^\omega$, where $k \ge 2$, is denoted by $\DNC_k$.
\end{definition}

In several of the theorems, $h$ will in addition be nondecreasing and unbounded:

\begin{definition}
	An \emph{order function} is a computable, nondecreasing, and unbounded $h \in \omega^\omega$ such that for all $n \in \omega$, $h(n) \ge 2$.
\end{definition}

\begin{definition} \label{def:bushiness}
	Given $\sigma \in \bstrings$, we say that a tree $T \subseteq \bstrings$ is \emph{$n$-bushy above $\sigma$} if every element of $T$ is comparable with $\sigma$, and for every $\tau \in T$ that extends $\sigma$ and is not a leaf of $T$, $\tau$ has at least $n$ immediate extensions in $T$. We will refer to $\sigma$ as the \emph{stem} of $T$.
\end{definition}
Note that under this definition, the set of initial segments of $\sigma$ is actually $n$-bushy above $\sigma$.

Suppose $\sigma \in \bstrings$ can be extended to a $\DNC$ function. In other words, for all $e < |\sigma|$, the $e^\text{th}$ entry of $\sigma$ does not equal $\varphi_e(e)$ when it is defined. The basic motivation behind Definition~\ref{def:bushiness} is that any tree that is $2$-bushy above $\sigma$, by always containing at least two immediate extensions of any non-leaf string, allows one to avoid the values of the diagonal partial computable function, and therefore has a path in it that extends $\sigma$ and (if finite) can itself be extended to a $\DNC$ function. 

\begin{definition} \label{def:bigness}
	Given $\sigma \in \bstrings$, we say that a set $B \subseteq \bstrings$ is \emph{$n$-big above $\sigma$} if there is a finite $n$-bushy tree $T$ above $\sigma$ such that all its leaves are in $B$. If $B$ is not $n$-big above $\sigma$ then we say that $B$ is \emph{$n$-small} above $\sigma$.
\end{definition}

Let $B_\DNC \subseteq \bstrings$ denote the set of strings that cannot be extended to a $\DNC$ function. Using the terminology established in Definition~\ref{def:bigness}, the observation immediately preceding it can be rephrased as follows: $B_\DNC$ is $2$-small above any $\sigma \notin B_\DNC$.

We begin by establishing some of the basic combinatorial properties of bushy trees. The first is that we can extend the leaves of an $n$-bushy tree with $n$-bushy trees to obtain another $n$-bushy tree (the proof is immediate, hence omitted):

\begin{lemma}[Concatenation property]\label{lem:concatenation}
	Suppose that $A \subseteq \bstrings$ is $n$-big above $\sigma$. If $A_\tau \subseteq \bstrings$ is $n$-big above $\tau$ for every $\tau \in A$, then $\bigcup_{\tau \in T} A_\tau$ is $n$-big above $\sigma$.
\end{lemma}

The second property that we use frequently is known as the \emph{smallness preservation property}. This is the \emph{second sparse subset property} of Kumabe and Lewis \cite{KumabeLewis}, and Lemma 5.4 of Greenberg and Miller \cite{MillerGreenberg}.
\begin{lemma}[Smallness preservation property]\label{lem:big-subset}
	Suppose that $B$ and $C$ are subsets of $\bstrings$, that $m, n \in \omega$ and that $\sigma \in \bstrings$. If $B$ and $C$ are respectively $m$-small and $n$-small above $\sigma$ then $B \cup C$ is $(n+m-1)$-small above $\sigma$.
\end{lemma}
\begin{proof}
	Let $T$ be an $(m + n - 1)$-bushy tree above $\sigma$ with leaves in $B \cup C$. We show that either $B$ is $m$-big above or $C$ is $n$-big above $\sigma$. Label a leaf $\tau$ of $T$ ``B'' if it is in $B$, ``C'' otherwise. Now if $\rho$ is the immediate predecessor of $\tau$, then $\rho$ has at least $(m + n - 1)$ immediate extensions on $T$, each of which are labeled either ``B'' or ``C''. Then either $m$ of these are labeled ``B'', in which case we label $\rho$ ``B'', or $n$ are labeled ``C'', in which case we label $\rho$ ``C''. Continuing this process leads to $\sigma$ eventually getting a label. It is clear that if $\sigma$ is labeled ``B'' then $B$ is $m$-big above $\sigma$. Otherwise $C$ is $n$-big above $\sigma$.
\end{proof}

The third property is known as the \emph{small set closure property}:

\begin{lemma}[Small set closure property]\label{lem:small-set-closure}
	Suppose that $B \subset \bstrings$ is $k$-small above $\sigma$. Let $C = \dset{\tau \in \bstrings}{\text{$B$ is $k$-big above $\tau$}}$. Then $C$ is $k$-small above $\sigma$. Moreover $C$ is \emph{$k$-closed}, meaning that if $C$ is $k$-big above a string $\rho$, then $\rho \in C$.
\end{lemma}
\begin{proof}
	Suppose that $C$ is $k$-big above a string $\rho$. Then, since $B$ is $k$-big above every $\tau \in C$, by the concatenation property, $B$ is $k$-big above $\rho$, so $\rho \in C$. The lemma follows immediately.
\end{proof}

The small set closure property is quite useful in the context of a forcing construction. Typically, $\sigma$ is an approximation to a function that we are building and $B$ is a set of strings that must be avoided in order to ensure that requirements remain met. We refer to it as the ``bad set''. Throughout the construction, we may wish to maintain the property that the bad set $B$ is $k$-small above $\sigma$ for some $k \in \omega$. Now, if $B$ is $k$-big above some string $\rho$, then $\rho$ is off-limits as well. Lemma \ref{lem:small-set-closure} allows us to assume that all such strings are already in the bad set, while preserving its smallness. From now on, whenever we deal with a bad set that is $k$-small, we also assume that it is $k$-closed. Note that the $k$-closure of a c.e.\ set of strings is also c.e.

\section{Basic bushy forcing}

As a first illustration of the convenience afforded us by these lemmas, we present a proof of a well-known result. Any bounded $\DNC$ function (i.e., a function in $\DNC_k$ for some $k \ge 2$) computes a function in $\DNC_2$. However, Jockusch showed in \cite{Jockusch} that this is not uniform.

\begin{theorem}[Jockusch \cite{Jockusch}]\label{thm:nonuniform}
	For each $n \ge 2$, there is no single functional $\Gamma$ such that for all $f \in \DNC_{n+1}$, $\Gamma^f \in \DNC_n$.
\end{theorem}
\begin{proof}
	Let us assume that such a $\Gamma$ exists, i.e., for all $f \in \DNC_{n+1}$, $\Gamma^f \in \DNC_n$. The set of sequences in $\DNC_{n+1}$ is a $\Pi^0_1$ subset of $(n+1)^\omega$. It is well known that a functional that is total on a $\Pi^0_1$ subset of $k^\omega$ can be modified to obtain one that agrees with it on the $\Pi^0_1$ subset and which is total on $k^\omega$. Let $\Xi$ be so obtained from $\Gamma$. We may also assume that $\Xi^f \in n^\omega$ for all $f \in (n+1)^\omega$. 
	
	For each $m \in \omega$ and for each $i < n$, let $\Lambda_{i, m} = \dset{\sigma \in (n+1)^{<\omega}}{\Xi^\sigma(m) = i}$. By the compactness of $(n+1)^\omega$, there exists a finite level $k$ such that for every string $\tau \in (n+1)^k$, $\Xi^\tau(m)$ converges. Therefore, $\bigcup_{i < n}\Lambda_{i, m}$ is $(n+1)$-big above the empty string $\estr$. It is now easy to see, by repeatedly applying the smallness preservation property, that for some $i < n$, $\Lambda_{i, m}$ must be $2$-bushy above $\estr$.

	We specify a partial computable function $\varphi$. On input $m$, $\varphi$ searches for a $2$-bushy tree $T$ above $\estr$ such that for every leaf $\tau$ of $T$, $\Xi^\tau(m)$ converges to the same value $i$, which it then outputs. By the argument above, such a tree must exist, and so $\varphi(m)$ is defined for each $m$. Let $e$ be the index for $\varphi$, and let $T_e$ be the $2$-bushy tree that $\varphi$ finds on input $e$. 

	As we have observed, $B_\DNC$ is $2$-small above $\estr$, and so there is a leaf $\tau$ of $T_e$ that can be extended to an $f \in (n+1)^\omega$ that is $\DNC_{n+1}$. But then $\Xi^f(e) = \Xi^\tau(e) = \varphi_e(e)$, which is a contradiction.
\end{proof}

Finitely iterating this strategy yields the following stronger result:

\begin{theorem}
	For each $n \ge 2$, there is no finite set of functionals $\Gamma_0, \Gamma_1, ..., \Gamma_k$  such that for all $f \in \DNC_{n+1}$, there exists a $j \le k$ such that $\Gamma_j^f \in \DNC_n$.
\end{theorem}
\begin{proof}
	Let us assume that such a set of functionals exists. We define a new functional $\Xi$ as follows: on input $e$, $\Xi$ simulates $\Gamma_0$ through $\Gamma_k$ on input $e$ and outputs the result of whichever one converges first. We may again assume, without loss of generality, that $\Xi$ is total on $(n+1)^\omega$. We then proceed exactly as in the proof of Theorem~\ref{thm:nonuniform}, obtaining a string $\sigma_0$ that is $\DNC_{n+1}$ and an $e \in \omega$ such that $\Xi^{\sigma_0}(e) = \varphi_e(e)$. Then $\Xi^{\sigma_0}(e) = \Gamma_j^{\sigma_0}(e)$ for some $j \le k$. It follows that $\Gamma_j$ fails to compute a $\DNC_n$ function on any $f \in \DNC_{n+1}$ extending $\sigma_0$. We now repeat the same process above $\sigma_0$ with the reduced list of functionals $\{\Gamma_1, ..., \Gamma_k\} \setminus \{\Gamma_j\}$, obtaining a $\DNC_{n+1}$ string $\sigma_1$ extending $\sigma_0$ that diagonalizes against one of the remaining functionals. After $k+1$ iterations, we will have obtained a contradiction.
\end{proof}

The previous proof points the way towards more sophisticated constructions involving bushy trees where we satisfy countably many requirements. The next result is our first example of such a construction. It features a simpler variant of bushy tree forcing, which we term \emph{basic bushy forcing}. In this type of forcing, the approximation to the generic object is a finite string.  

\begin{theorem}[Ambos-Spies, Kjos-Hanssen, Lempp, and Slaman \cite{DNRWWKL}]\label{thm:AKLS1}
	There is a $\DNC$ function that computes no computably bounded $\DNC$ function.
\end{theorem}
\begin{proof}
	
The forcing conditions are pairs $(\sigma, B)$, where $\sigma \in \bstrings$, $B \subset \bstrings$ and:

\begin{itemize}
	\item for some $k \in \omega$, $B$ is $k$-small above $\sigma$ (and without loss of generality, $k$-closed)
	\item $B$ is \emph{upward closed} (i.e., if $\gamma$ is in $B$, then all extensions of $\gamma$ are in $B$).
\end{itemize}

The string $\sigma$ is an approximation to $f$ and the set $B$ is a ``bad set'', i.e., a set of strings that must be avoided in order to ensure that requirements remain satisfied. 

A condition $(\sigma, B)$ \emph{extends} another condition $(\tau, C)$ if $\tau \preceq \sigma$ and $C \subseteq B$. Let $\pP$ denote this partial order. Now if $\+G$ is a filter on $\pP$, then for any two elements $(\sigma, B)$ and $(\tau, C)$ of $\+G$, $\sigma$ and $\tau$ are comparable. Hence, $f_{\+G} = \bigcup \dset{\sigma}{(\sigma, B) \in \+G} \in \omega^{\le \omega}$. In fact, we can ensure that $f_\+G$ is total:

\begin{claim}\label{lem:AKLS1-totality}
	If $\+G$ is sufficiently generic with respect to $\pP$, then $f_\+G$ is total.
\end{claim}
\begin{proof} 
	We show that the collection $\+T_m = \dset{(\sigma, B) \in \pP}{|\sigma| \ge m}$ is dense in $\pP$. Suppose $(\sigma, B) \in \pP$, where $|\sigma| < m$. Then $B$ is $k$-small above $\sigma$ for some $k \in \omega$. The set $C = \dset{\tau \in \bstrings}{|\tau| \ge m}$ is $k$-big above $\sigma$, so let $\tau$ be any string in $C \setminus B$. Then $(\tau, B) \in \pP$.
\end{proof}

\begin{claim}\label{lem:AKLS1-bad-set-avoidance}
	If $\+G$ is any filter on $\pP$, then for all $(\sigma, B) \in \+G$, $f_\+G$ has no initial segment in $B$.
\end{claim}
\begin{proof}
	Suppose that $f_\+G$ has an initial segment $\tau$ in $B$. Then there is a $(\rho', C') \in \+G$ such that $\rho'$ extends $\tau$. Let $(\rho, C)$ be a common extension of $(\rho', C')$ and $(\sigma, B)$. Since $B$ is upward closed, $\rho \in B$. But $B \subseteq C$, so $\rho \in C$. This is a contradiction, since it follows that $C$ is $k$-big above $\rho$ for all $k \in \omega$.
\end{proof}

If $\Gamma$ is a functional and $h$ a computable function such that $\Gamma$ is $h$-valued (in other words, whenever $\Gamma$ converges with any oracle on input $e$, its output is less than $h(e)$), let $\+D_{\Gamma, h}$ denote the set of $(\sigma, B) \in \pP$ such that for all $g \in [\sigma] \setminus [B]^\prec$, $\Gamma^g$ is not a $\DNC_{h}$ function. 

\begin{claim}\label{lem:AKLS1-density}
	For each computable function $h$, and $h$-valued functional $\Gamma$, $\+D_{\Gamma, h}$ is dense in $\pP$.
\end{claim}
\begin{proof}
	Suppose $(\sigma, B) \in \pP$ and that $B$ is $k$-small above $\sigma$. As in the proof of Theorem~\ref{thm:nonuniform}, we specify a partial computable function $\varphi$. On input $m$, $\varphi$ searches for a $k$-bushy tree $T$ above $\sigma$ such that for every leaf $\tau$ of $T$, $\Gamma^\tau(m)$ converges to the same value $i < h(m)$. Upon finding such a tree, $\varphi$ outputs $i$. Let $e$ be the index of $\varphi$.

	There are now two cases. If the set $A = \dset{\tau}{\Gamma^\tau(e)\downarrow}$ is $(h(e)\cdot k)$-small above $\sigma$, then $A \cup B$ is $(h(e)\cdot k + k-1)$-small above $\sigma$. Then $(\sigma, A \cup B) \in \pP$ and extends $(\sigma, B)$. Note that we have forced $\Gamma$ to be partial on any $g \in [\sigma] \setminus [A \cup B]^\prec$. Hence, $(\sigma, A \cup B) \in \+D_{\Gamma, h}$.
	
	On the other hand, if $A$ is $(h(e)\cdot k)$-big above $\sigma$, then for some $i < h(e)$, $\dset{\tau}{\Gamma^\tau(e)\downarrow = i}$ is $k$-big above $\sigma$. So $\varphi(e)$ is defined. In this case, we extend $\sigma$ to any $\tau$ not in $B$ such that $\Gamma^\tau(e)\downarrow = \varphi(e)$. This forces  $\Gamma^g$ to fail to be $\DNC$ on any $g$ extending $\tau$. Hence, $(\tau, B) \in \+D_{\Gamma, h}$.
\end{proof}
	
	Finally, $B_\DNC$, the set of finite strings that cannot be extended to a $\DNC$ function, is $2$-small above $\estr$, so $(\estr, B_\DNC) \in \pP$. Let $\+G$ be a filter on $\pP$ containing $(\estr, B_\DNC)$ that meets $\+T_m$ for every $m \in \omega$ and $\+D_{\Gamma, h}$ for every computable function $h$ and $h$-valued functional $\Gamma$ (note that this is a countable collection of dense sets). 
	
	By Claim~\ref{lem:AKLS1-totality}, $f_\+G$ is total. By Claim~\ref{lem:AKLS1-bad-set-avoidance} and the fact that $(\estr, B_\DNC) \in \+G$, $f_\+G$ is a $\DNC$ function. If $f_\+G$ computes a function in $\DNC_h$ for some computable function $h$, then it does so via an $h$-valued functional $\Gamma$. Claim~\ref{lem:AKLS1-density} shows that this is not the case. This concludes the proof of Theorem~\ref{thm:AKLS1}.\end{proof}
	
 	We note that while the bad sets in the previous proof are c.e., we do not make use of this fact. Given an oracle $X$, let $B^X_\DNC$ denote the set of finite strings that are not $\DNC$ \emph{relative to $X$}. Note that $B^X_\DNC$ is not necessarily c.e., but is nevertheless $2$-small above $\estr$. This suggests that we could use the same sort of techniques to construct a function that is $\DNC$ relative to $X$. As an example, we prove a theorem that implies the main result in \cite{DNRWWKL}, and is slightly stronger.

	\begin{theorem}\label{thm:computes-no-DNC-h}
		Fix a computable function $h$. Suppose $X$ computes no $\DNC_h$ function. Then there is an $f$ that is $\DNC$ relative to $X$ such that $f \oplus X$ computes no $\DNC_h$ function.
	\end{theorem}
	
	\begin{proof}
		The forcing partial order is the same as before. If $\Gamma$ is an $h$-valued functional, let $\+D_{\Gamma}$ denote the set of $(\sigma, B) \in \pP$ such that for all $f \in [\sigma] \setminus [B]^\prec$, $\Gamma^{f \oplus X}$ is not a $\DNC_{h}$ function. We show that $\+D_{\Gamma}$ is dense in the partial order. Suppose $(\sigma, B)$ is a condition where $B$ is $k$-small above $\sigma$.
		
		First, if there are $x, l \in \omega$ such that \[C_x = \dset{\tau \in \bstrings}{\Gamma^{\tau \oplus X}(x)\converges}\] is $l$-small above $\sigma$, then the condition $(\sigma, B \cup C_x)$ extends $(\sigma, B)$ and forces the divergence of $\Gamma^{f_\+G \oplus X}(x)$. Therefore, let us assume that for each $x, l \in \omega$, $C_x$ is $l$-big above $\sigma$.
		
		Next, if there exists an $x \in \omega$ such that $\varphi_x(x)$ converges and \[N_x = \dset{\tau \in \bstrings}{\Gamma^{\tau \oplus X}(x)\converges = \varphi_x(x)}\] is $k$-big above $\sigma$, then there is a $\tau$ extending $\sigma$ not in $B$ such that $\Gamma^{\tau \oplus X}(x) \converges = \varphi_x(x)$, and so the condition $(\tau, B)$ extends $(\sigma, B)$ and forces that $f_\+G$ is not $\DNC$. Therefore, let us assume that for each $x \in \omega$, either $\varphi_x(x)$ diverges or $N_x$ is $k$-small above $\sigma$.
		
		We now describe how to compute a $\DNC_h$ function from $X$, which yields a contradiction. On input $x$, search for a $k$-bushy tree $T$ above $\sigma$ such that for every leaf $\tau$ of $T$, $\Gamma^{\tau \oplus X}(x)$ converges to the same value $j < h(x)$, then output $j$. Since for each $x$, $C_x$ is $(h(x)\cdot k)$-big above $\sigma$, such a tree $T$ exists. So the $X$-computable function just described is total. Moreover, it disagrees with $\varphi_x(x)$ whenever it is defined, since $N_x$ is $k$-small above $\sigma$.
		
		Therefore, $\+D_{\Gamma}$ is dense. Let $\+G$ be a generic filter including the condition $(\estr, B^X_{\DNC})$. Then $f_\+G$ has the required properties.
		\end{proof}

	With a stronger assumption, the technique in the proof of Theorem~\ref{thm:computes-no-DNC-h} yields a stronger conclusion: If $X$ computes no computably bounded $\DNC$ function, then there is an $f$ that is $\DNC$ relative to $X$ such that $f \oplus X$ computes no computably bounded $\DNC$ function. We omit the proof.
	
	An analysis of the amount of bushiness we require above $\sigma$ in the diagonalization argument of Claim~\ref{lem:AKLS1-density} yields the following:
	
\begin{theorem}[Ambos-Spies, et al. \cite{DNRWWKL}]\label{thm:AKLS2}
	For each order function $h$ there is an order function $j$ and a function $f \in \DNC_j$ that computes no function in $\DNC_h$.
\end{theorem}
\begin{proof}
If $j$ is an order function, let $j^n$ denote the space
	\[ \prod_{m < n} \{0, 1, ..., j(m) - 1\},\]
and let $j^{<\omega}$ and $j^\omega$ be defined in the obvious way.
	
We now fix a computable function $h$ and let $(\Gamma_i)_{i \in \omega}$ be an effective enumeration of all $h$-valued Turing functionals. We define an order function $j$ by recursion. In order to define $j$, we will also define an auxiliary computable function $q:\bstrings \times \omega^2$, the definition of which will refer to the index of the function $j$. This is possible because we can assume, by the recursion theorem, that we have access to the index of $j$ in advance.  

On input $x$, $\varphi_{q(\sigma, i)}$ searches for a $|\sigma|$-bushy tree $T$ above $\sigma$ contained in $j^{<\omega}$ such that for every leaf $\tau$ of $T$, $\Gamma_i^\tau(x)$ converges to the same value $k < h(x)$, and upon finding such a tree, itself outputs $k$. Now let $\bar{q} = \max_{i < n, \sigma \in j^n} q(\sigma, i)$. We define $j(n)$ to be the larger of $\max_{i < n} j(i)$ and $((h(\bar{q}(n)) + 1) \cdot n) + 2$.

The forcing conditions are now pairs $(\sigma, B)$ where $B \subseteq j^{<\omega}$ and $\sigma \in j^{<\omega} \setminus B$. We require that $B$ be upward-closed and $|\sigma|$-small above $\sigma$. By the small set closure property, we may assume that $B$ is $|\sigma|$-closed. For $\sigma \in j^{<\omega}$, let $[\sigma]_j$ denote $\dset{X \in j^\omega}{\sigma \prec X}$.

\begin{claim}\label{lem:AKLS2-density}
Let $\+D_i$ denote the set of $(\sigma, B) \in \pP$ such that for all $g \in [\sigma]_j \setminus [B]^\prec$, $\Gamma_i^g$ is not a $\DNC_{h}$ function. Then for each $i \in \omega$, $\+D_i$ is dense in $\pP$.
\end{claim}
\begin{proof}
	Suppose that $(\sigma, B) \in \pP$. By suitably extending $\sigma$, we can assume that $|\sigma| > i$. Let $n = |\sigma|$ and \[ A = \dset{\tau \in j^{<\omega}}{\Gamma_i^\tau(q(\sigma, i))\downarrow}. \]
	As in the proof of Claim~\ref{lem:AKLS1-density}, there are two cases.
	
	If $A$ is $(h(q(\sigma, i)) \cdot n)$-small above $\sigma$, then letting $c = (h(q(\sigma, i)) \cdot n + n - 1)$, $A \cup B$ is $c$-small above $\sigma$. Let $C$ be the $c$-closure of $A \cup B$. Since $j(n) \ge (h(q(\sigma, i)) + 1) \cdot n > c$ and $j$ is nondecreasing, $j^c$ is $c$-big above $\sigma$. Let $\tau$ be any string extending $\sigma$ in $j^c \setminus C$. Then $(\tau, C)$ is a condition. Further, $\Gamma_i^f$ is partial on any $f \in [\tau]_j \setminus [C]^\prec$, so $(\tau, C) \in \+D_i$.
	
	On the other hand, if $A$ is $(h(q(\sigma, i)) \cdot n)$-big above $\sigma$, then for some $k < h(q(\sigma, i))$, the set $\dset{\tau \in j^{<\omega}}{\Gamma_i^\tau(q(\sigma, i)) \downarrow = k}$ is $n$-big above $\sigma$. It follows that $\varphi_{q(\sigma, i)}(q(\sigma, i))$ is defined. So there is a $\tau \in j^{<\omega} \setminus B$ extending $\sigma$ such that $\Gamma_i^\tau(q(\sigma, i))) = \varphi_{q(\sigma, i)}(q(\sigma, i))$. Then $(\tau, B) \in \pP \cap \+D_i$.
\end{proof}
This concludes the proof of Theorem~\ref{thm:AKLS2}.\end{proof}

\begin{theorem}\label{thm:DNC-below}
	Given any order function $g$, there is an order function $h$ and an $f \in \DNC_g$ such that $f$ computes no $\DNC_h$ function.
\end{theorem}
\begin{proof}
	We define $h$ inductively. Let $n_0 = 0$ and let $h(0) = 2$. At the $i^\text{th}$ stage of the construction, suppose we have defined it up to $n_i$. Let $k \ge n_i + 1$ be the least such that $g(k) \ge (h(n_i) + 1)\cdot g(n_i)$. Let $q(\sigma)$ be the computable function such that if $\sigma \in g^k$, then $q(\sigma) \ge k$, and $\varphi_{q(\sigma)}(n)$ searches for a $g(n_i)$-bushy tree $T$ above $\sigma$ contained in $g^{<\omega}$ such that for every leaf $\tau$ of $T$, $\Phi_{i-1}^\tau$ converges to the same value $l < h(n_i)$. Let $m = \max_{\sigma \in g^k} q(\sigma)$. Let $h(n) = h(n_i)$ for all $n$ such that $n_i < n \le m$ and let $h(m+1) = h(m) + 1$. Finally, let $n_{i+1} = m+1$, ensuring that $h$ is unbounded. The fact that $k \ge n_i + 1$ ensures that $h$ is total.

	It remains to construct $f$. Let $B_0 = B_{\DNC}$ and let $\sigma_0 \in g^{1} \setminus B_{\DNC}$. Assume inductively that $\sigma_i \in g^{n_i} \setminus B_i$ and that $B_i$ is $g(n_i)$-small above $\sigma_i$. Let $k$ and $q$ be defined as above and extend $\sigma$ to a string $\rho \in g^k \setminus B_i$. For $j < h(q(\rho))$, let \[A_j = \dset{\tau \in g^{<\omega}}{\Phi^\tau_i(q(\rho))\converges = j}.\]
	If $A_j$ is $g(n_i)$-big above $\rho$ for some $j$, then $\varphi_{q(\rho)}(q(\rho))$ is defined. If $\varphi_{q(\rho)}(q(\rho)) = j'$ then there is a $\tau \in A_{j'} \setminus B_i$ extending $\rho$ such that $\Phi_{i-1}^\tau(q(\rho)) = \varphi_{q(\rho)}(q(\rho))$. Otherwise, $C = (\bigcup_{j < h(q(\rho))} A_j) \cup B_i$ is $(h(q(\rho)) + 1)\cdot g(n_i)$-small above $\rho$. Since $g(k) \ge (h(n_i) + 1) \cdot g(n_i) = (h(q(\rho)) + 1)\cdot g(n_i)$, $C$ is $g(k)$-small above $\rho$. So let $B_{i+1} = C$ and let $\sigma_{i+1}$ be any string in $g^{n_{i+1}} \setminus B_{i+1}$ extending $\rho$.
	Finally, let $f = \bigcup_{i \in \omega} \sigma_i$.
\end{proof}

By alternating the strategies of Theorems~\ref{thm:AKLS2} and \ref{thm:DNC-below}, one can also show:
\begin{theorem}
	Given any order function $g_0$, there is another order function $g_1$ and functions $f_0 \in \DNC_{g_0}$ and $f_1 \in \DNC_{g_1}$ such that $f_0$ computes no $\DNC_{g_1}$ function and $f_1$ computes no $\DNC_{g_0}$ function.
\end{theorem}

\section{Bushy tree forcing}
Bounded $\DNC$ functions, being of PA degree, compute \ML\ random reals. Ku\v{c}era \cite{Kucera} showed that there is an order function $h$ such that every \ML\ random real computes a $\DNC_h$ function. Theorem~\ref{thm:AKLS1} then implies that there are unbounded $\DNC$ functions that compute no \ML\ random real. Greenberg and Miller established a stronger version of this fact:

\begin{theorem}[Greenberg and Miller \cite{MillerGreenberg}]\label{thm:slow-no-random}
For each order function $h$, there is an $f \in \DNC_h$ that computes no \ML\ random real.
\end{theorem}

The proof uses basic bushy forcing, and does not require that the bad sets be c.e. In fact, the same technique could be used to show that for each order function $h$ and each oracle $X$, there is an $f \in \DNC^X_h$ that computes no \ML\ random real. Our main result in this section cannot be partially relativized in this manner (it strongly depends on the fact that the bad sets are c.e.) but improves upon the Greenberg-Miller theorem in a different way. Recall that a real is \emph{Kurtz random} (sometimes also called \emph{weakly random}) if it is not contained in any measure $0$ $\Pi^0_1$ class.

\begin{theorem} \label{thm:slow-no-Kurtz}
For each order function $h$, there is an $f \in \DNC_h$ that computes no Kurtz random real.
\end{theorem}

Theorem~\ref{thm:slow-no-Kurtz} is our first example of \emph{bushy tree forcing}, where the conditions consist of trees, not just finite strings. The atomic step in the forcing is based on the following result of Downey, Greenberg, Jockusch, Milans \cite{DowneyJockusch}, which we prove here for convenience.

\begin{theorem}[Downey, et al.\ \cite{DowneyJockusch}]\label{thm:no-Kurtz}
There is no single functional $\Gamma$ such that $\Gamma^f$ is Kurtz random for all $f \in \DNC_3$.
\end{theorem}
\begin{proof}
	Suppose that such a functional $\Gamma$ exists. As before, we may assume that $\Gamma$ is total. It will be convenient to assume that $\Gamma$ satisfies the following additional property:
	\begin{itemize}
		\item If $\sigma \in 3^{<\omega}$ and $\Gamma^\sigma(n)$ converges, then $\Gamma^\sigma(n)$ converges within $|\sigma|$ steps and for all $n' < n$, $\Gamma^\sigma(n')$ also converges.
	\end{itemize}
	
	It is not difficult to see that this assumption can be made without any loss of generality and that if $\Gamma$ satisfies this property, then $\Gamma^\sigma = \tau$ is a computable relation for $\sigma \in 3^{<\omega}$ and $\tau \in 2^{<\omega}$.
	
	We build a computable $2$-bushy subtree $S$ of $3^\omega$ with no leaves such that the image of $\Gamma$ on $S$ (denoted by $\Gamma(S)$) has measure $0$. The tree $S$ will be obtained as the union of a sequence $\{\estr\} = S_0 \subset S_1 \subset S_2 ... $ of finite regular\footnote{All the leaves are of the same length.} binary subtrees of $3^{<\omega}$. Let $\Gamma(S_i)$ denote the set of reals \[ \bigcup\dset{[\Gamma^\sigma]}{\sigma \text{ is a leaf of $S_i$}}.\]
	
	In constructing $S_{i+1}$, we want to ensure that $\mu(\Gamma(S_{i+1})) \le (3/4)\mu(\Gamma(S_i))$. Let $L = \{\sigma_0, \sigma_1, ..., \sigma_{|L| - 1}\}$ be the set of leaves of $S_i$ and let $m = \max\dset{|\Gamma^\sigma|}{\sigma \in L}$. Our assumption on $\Gamma$ above allows us to find $m$ computably. Let $l$ be large enough so that for all $\tau \in 3^l$, $|\Gamma^\tau| \ge m + (2^{|L|} + 1)$. In other words, $l$ is large enough so that we obtain at least $2^{|L|} + 1$ \emph{additional} bits of convergence by extending a leaf of $S_i$ to any ternary string of length $l$. Note that such an $l$ exists by the compactness of $3^\omega$ and that we can find it computably. Let $T_j = \dset{\tau \in 3^l}{ \tau \succ \sigma_j}$.
	
	Suppose that $k$ is a position corresponding to one of the additional bits of convergence, i.e., $m \le k < m + 2^{|L|} + 1$. Since each $T_j$ is $3$-big above $\sigma_j$, by the smallness preservation property, either $\dset{\tau \in T_j}{\Gamma^\tau(k) = 1}$ is $2$-big above $\sigma_j$ (in which case, we say that we can \emph{force} the $k^\tH$ bit to be $1$ above $\sigma_j$) or $\dset{\tau \in T_j}{\Gamma^\tau(k) = 0}$ is $2$-big above $\sigma_j$ (we say that we can force the $k^\tH$ bit to be $0$ above $\sigma_j$). This allows us to obtain a binary sequence $\rho_k$ of length $|L|$, where $\rho_k(j) = 1$ if we can force the $k^\tH$ bit to be $1$ above $\sigma_j$, and $0$ otherwise. Moreover, we can computably find $2$-big sets above $\sigma_j$ that force the $k^\tH$ bit one way or another, so we can compute $\rho_k$, given $k$.
	
	By the pigeonhole principle, there exist $r$ and $s$ such that $m \le r, s < m + 2^{|L|} + 1$ and $\rho_r = \rho_s$. Note that for each $j < |L|$, even though we can force the $r^\tH$ and $s^\tH$ bits in the same way above $\sigma_j$, we may not be able to do so \emph{simultaneously}. We adopt the following strategy above each $\sigma_j$: If we can force the $r^\tH$ bit to be $1$ above $\sigma_j$, we do so, by extending $\sigma_j$ to a finite $2$-bushy tree $B_j$ with leaves in $3^l$ such that for every leaf $\tau$ of $B_j$, $\Gamma^\tau(r) = 1$. Otherwise, $\rho_r(j) = \rho_s(j) = 0$, so we force the $s^\tH$ bit to be $0$ above $\sigma_j$. The regular binary tree of height $l$ that results is $S_{i+1}$.
	
	For any leaf $\tau$ of $S_{i+1}$, it is \emph{not} the case that the $r^\tH$ bit of $\Gamma^\tau$ is $0$ and the $s^\tH$ bit is $1$: Say $\tau$ extends $\sigma_j$. By our choice of strategy, if the $r^\tH$ bit is $0$, then it must be the case that we could not have forced it to be $1$ above $\sigma_j$, and so we would have forced the $s^\tH$ bit to be $0$ above $\sigma_j$. 
	
	Let $P = \dset{X \in \Gamma(S_i)}{X(r) = 0 \text{ and } X(s) = 1}$. Then $\mu(P) = (1/4)\mu(\Gamma(S_i))$, since $r, s \ge m$. Now, $\Gamma(S_{i+1}) \subseteq \Gamma(S_i) \setminus P$, so $\mu(\Gamma(S_{i+1})) \le (3/4)\mu(\Gamma(S_i))$, as desired.
	
	Let $S = \bigcup_{i \in \omega} S_i$. Then $\mu(\Gamma(S)) = \mu(\bigcap_{i \in \omega} \Gamma(S_i)) = 0$. Let $f$ be any path through $S$ that is $\DNC_3$. Then $\Gamma^f \in \Gamma(S)$. But $\Gamma(S)$ is a null $\Pi^0_1$ class, which implies that $\Gamma^f$ is not Kurtz random, contradicting our initial assumption.
\end{proof}

Note that the construction in Theorem~\ref{thm:no-Kurtz} starts with a $3$-bushy tree and produces a $2$-bushy subtree with no leaves. 

\begin{definition}
	Let $j$ be an order function. We say that a tree $T \subseteq \bstrings$ is \emph{$j$-bushy} above a string $\sigma \in \bstrings$ if every element of $T$ is comparable with $\sigma$ and for each $\tau$ extending $\sigma$ that is not a leaf of $T$, there are at least $j(|\tau|)$ many immediate extensions of $\tau$. We say $T$ is \emph{exactly $j$-bushy above $\sigma$} if for each nonleaf $\tau$, there are exactly $j(|\tau|)$ immediate extensions of $\tau$ in $T$.
\end{definition}

\begin{proof}[Proof of Theorem~\ref{thm:slow-no-Kurtz}]
The forcing conditions have the form $(\sigma,T,B)$, where $\sigma \in \bstrings$, $T$ is a computable subtree of $\bstrings$, $B \subset T$ and:

\begin{itemize}
	\item $T$ is exactly $j$-bushy above $\sigma$ for some order function $j$,
	\item $B$ is c.e. and upward-closed in $T$ (i.e., if $\tau \in B$ then $\rho$ extending $\tau$ on $T$ is also in $B$),
	\item $B$ is $j(|\sigma|)$-small above $\sigma$ (and, without loss of generality, $j(|\sigma|)$-closed).
\end{itemize}

A condition $(\sigma,T,B)$ \emph{extends} another condition $(\tau,S,C)$ if $\sigma\succeq\tau$, $T\subseteq S$ and $B \cap T \supseteq C \cap T$. Let $\pP$ denote this partial order. As before, if $\+G$ is a filter on $\pP$, then $f_{\+G} = \bigcup \dset{\sigma}{(\sigma, T, B) \in \+G} \in \omega^{\le \omega}$. It is not difficult to verify that if $\+G$ is sufficiently generic, then $f_\+G$ is total and if $(\sigma, T, B) \in \+G$, then $f_\+G$ contains no initial segment in $B$.

If $\Gamma$ is any functional, let $\+D_\Gamma$ denote the set of $(\sigma, T, B) \in \pP$ such that either
\begin{itemize}
\item $g \in [T]\smallsetminus [B]^\prec$ implies that $\Gamma^g$ is total, or
\item there is an $n\in\omega$ such that $g \in [T]\smallsetminus [B]^\prec$ implies that $\Gamma^g(n)\diverges$.
\end{itemize}

\begin{claim}\label{lem:forcing-totality}
$\+D_\Gamma$ is dense in $\pP$.
\end{claim}

\begin{proof}
	Suppose $(\sigma, T, B) \in \pP$, where $T$ is exactly $j$-bushy above $\sigma$. Let $C_x = \dset{\tau \in T}{\Gamma^\tau(x)\converges}$. Note that $C_x$ is c.e.\ and upward closed in $T$. As usual, there are two cases.
	
	\emph{Case 1.} For every $\tau \in T$ extending $\sigma$ and every $x \in \omega$, $C_x \cup B$ is $j(|\tau|)$-big above $\tau$. In this case, we build a computable tree $S \subseteq T$ in stages that is exactly $j'$-bushy above $\sigma$ for an order function $j'$. Let $S_0$ consist of just $\sigma$ and its initial segments. Suppose inductively that we have $l_i \in \omega$ and $S_i \subset T$ such that:
	\begin{itemize}
		\item For each $x < l_i$, $j'(x)$ has already been defined and $j'(x) \le j(x)$.
		\item $S_i$ is a finite, regular $j'$-bushy tree of height $l_i$ above $\sigma$.
		\item For every leaf $\tau$ of $S_i$, either $\Gamma^\tau(x)\converges$ for every $x < i$ or $\tau \in B$.
	\end{itemize}
	
	Let $\tau$ be a leaf of $S_i$. By assumption, $C_i \cup B$ is $j(|\tau|)$-big above $\tau$, so we extend $\tau$ to a finite tree with leaves in $C_i \cup B$ that is $j(|\tau|)$-bushy above $\tau$. Note that since $C_i \cup B$ is c.e., we can find such a tree computably. The tree $S'_{i+1}$ that results from carrying out this operation above each leaf of $S_i$ may not be regular, but since both $C_i$ and $B$ are upward closed in $T$ and $T$ is $j$-bushy above the leaves of $S'_{i+1}$, we can extend them $j(l_i)$-bushily to some common level $l_{i+1}$, retaining the property that every leaf is in $C_i$ or in $B$, and producing the tree $S_{i+1}$. We now let $j'(x) = j(l_i)$ for $l_i \le x < l_{i+1}$. Note that $j'$ is nondecreasing because of our assumption that $j'(x) \le j(x)$ for $x < l_i$.
	
	Let $S = \cup_{i \in \omega} S_i$ and note that since $j'(|\sigma|) = j(|\sigma|)$, $B$ is already $j'(|\sigma|)$-closed. So the condition $(\sigma, S, B \cap S)$ extends $(\sigma, T, B)$. Finally, if $g \in [S]\smallsetminus [B]^\prec$, then for every $i$, $g \res l_i \in C_i$, so $\Gamma^g$ is total. 

	 \emph{Case 2.} Let $\tau$ and $x$ be counterexamples to the assumption in Case 1 and let $S$ be the full subtree of $T$ above $\tau$. Let $B' = (C_x \cup B) \cap S$. Then $B'$ is $j(|\tau|)$-small above $\tau$, so $(\tau, S, B') \in \pP$ and if $g \in [S]\smallsetminus [B']^\prec$, then $\Gamma^g(x)$ diverges.
\end{proof}

Let $\+H_\Gamma$ be the set of all conditions $(\sigma, T, B)$ such that if $g \in [T] \smallsetminus [B]^\prec$, then $\Gamma^g$ is not Kurtz random.

\begin{claim}
	$\+H_\Gamma$ is dense in $\pP$.
\end{claim}
\begin{proof}
	Let $(\sigma, T, B) \in \pP$ and $\Gamma$ be a $\{0, 1\}$-valued functional. Claim~\ref{lem:forcing-totality} allows us to assume that $\Gamma$ is total on $[T]\smallsetminus[B]^\prec$, and since $B$ is c.e., we can assume further that $\Gamma$ is total on $[T]$. Let $j$ be the order function such that $T$ is exactly $j$-bushy above $\sigma$.
	
	The remainder of the proof is a straightforward modification of Theorem \ref{thm:no-Kurtz}. We build an order function $j'$ and an exactly $j'$-bushy tree $S \subseteq T$ above $\sigma$ in stages. Let $S_0$ consist of $\sigma$ and its initial segments. Next, suppose inductively that we have $l_i \in \omega$ and $S_i \subset T$ such that
	\begin{itemize}
		\item for each $x < l_i$, $j'(x)$ has already been defined and $j'(x) \le j(x)$, and
		\item $S_i$ is a finite, regular $j'$-bushy tree of height $l_i$ above $\sigma$.
	\end{itemize}

	Let $\Gamma(S_i)$ denote $\dset{\Gamma^g}{g \in [T] \cap [S_i]^\prec}$.
	
 	We first extend $S_i$ $j(l_i)$-bushily within $T$ to a height $q > l_i$ such that $j(q) \ge 2j(l_i)$, obtaining the tree $S'_{i+1}$. This ensures that every level of $T$ above $q$ is $2j(l_i)$-big above each leaf of $S'_{i+1}$. Now, $\mu(\Gamma(S'_{i+1})) \le \mu(\Gamma(S_i))$. Let $L$ be the set of leaves of $S'_{i+1}$ and let $m = \max\dset{|\Gamma^\rho|}{\rho \in L}$. We choose $l_{i+1}$ large enough so that for every $\tau \in T$ of length $l_{i+1}$, $|\Gamma^\tau| \ge m + 2^{|L|} + 1$. Note that the fact that $T$ is exactly $j$-bushy ensures that we can find $l_{i+1}$ computably. 

For any leaf $\rho$ of $S'_{i+1}$, let $T_\rho$ be the set of strings of length $l_{i+1}$ in $T$ extending $\rho$. If $k$ is a position corresponding to one of the additional bits of convergence (i.e., $m \le k < m + 2^{|L|} + 1$), we say we can force the $k^\tH$ bit to be $c \in \{0, 1\}$ above $\rho$ if $\dset{\tau \in T_\rho}{\Gamma^\tau(k) = c}$ is $j(l_i)$-big above $\rho$. Since $T_\rho$ is $2j(l_i)$-big above $\rho$, if we cannot force the $k^\tH$ bit to be $0$ above $\rho$, we can force it to be $1$. 

As in the proof of Theorem \ref{thm:no-Kurtz}, we obtain positions $r$ and $s$ such that above each leaf of $S'_{i+1}$, the $r^\tH$ and $s^\tH$ bits can be forced in the same way. We adopt the same strategy as before for extending $S'_{i+1}$ to $S_{i+1}$ and ensuring that $\mu(\Gamma(S_{i+1})) \le (3/4)\mu(\Gamma(S_i))$. Finally, we let $j'(x) = j(l_i)$ for $l_i \le x < l_{i+1}$.

	Let $S = \bigcup_{i \in \omega} S_i$. Since $j'(|\sigma|) = j(|\sigma|)$, $B \cap S$ is $j'(|\sigma|)$-small above $\sigma$. So $(\sigma, S, B \cap S) \in \pP$ and since $\mu(\Gamma(S)) = \mu(\bigcap_{i \in \omega} \Gamma(S_i)) = 0$, $(\sigma, S, B \cap S) \in \+H_\Gamma$.
\end{proof}

To conclude the proof of Theorem \ref{thm:slow-no-Kurtz}, let $\+G$ be any filter containing $(\estr, h^{<\omega}, B_\DNC)$ that meets $\+H_\Gamma$ for each functional $\Gamma$ as well as the families of conditions that ensure totality. Then $f_\+G \in \DNC_h$ and does not compute a Kurtz random. \end{proof}

Every hyperimmune degree contains a Kurtz random \cite{Kurtz}, so the function we have constructed is hyperimmune-free. This is, in fact, a feature of the partial order:

\begin{claim}\label{cor:hyperimmune-free}
If $\+G$ is sufficiently generic, then $f_\+G$ has hyperimmune-free degree.
\end{claim}

\begin{proof}
	Suppose $\Gamma^{f_\+G}$ is a total function. Then if $(\sigma, T, B) \in \+G \cap \+D_\Gamma$, it must be the case that $\Gamma$ is total on $[T]\smallsetminus [B]^\prec$. Let $\Xi$ be the functional that on input $x$ and oracle $\tau \in T$, computes $\Gamma^\tau(x)$ until the computation converges or $\tau$ enters $B$. If the latter occurs first, then let $\Xi^\tau(x) = 0$. Now $\Xi$ is total on $[T]$ and agrees with $\Gamma$ on $[T]\smallsetminus [B]^\prec$.
	
	 Let $j$ be the order function such that $T$ is exactly $j$-bushy above $\sigma$. We define a computable function $m$ that majorizes $\Gamma^{f_\+G}$. To compute $m(i)$, search for a finite tree $S_i \subset T$ that is $j$-bushy above $\sigma$ such that for every leaf $\tau$ of $S_i$, $\Xi^\tau(i)\downarrow$. Note that such a finite tree must exist by the compactness of $[T]$ and we can find it computably since $T$ is computable. Now let $m(i)$ be the maximum of the values $\Xi^\tau(i)$ as $\tau$ ranges over the leaves of $S_i$.
	
	Since $T$ is exactly $j$-bushy above $\sigma$ and $S_i$ is a subtree of $T$ that is $j$-bushy above $\sigma$, $[T] \subseteq [S_i]^\prec$. So $f_\+G \in [S_i]^\prec$ and $\Gamma^{f_\+G}(i) = \Xi^{f_\+G}(i) \le m(i)$.
\end{proof}

	Theorem~\ref{thm:slow-no-Kurtz} yields an interesting corollary when combined with the following result:

	\begin{theorem}[Greenberg and Miller \cite{MillerGreenberg}] \label{thm:slowhdim1}
		There is an order function $h$ such that every $\DNC_h$ function computes a real of effective Hausdorff dimension 1.
	\end{theorem}

	There is a $\DNC_h$ function that computes no Kurtz random real, and so we have:

	\begin{corollary}
		There is a real of effective Hausdorff dimension 1 that computes no Kurtz random real.
	\end{corollary}

\section{A $\DNC^X$ function of minimal degree}

In this section, we strengthen Kumabe's result that there is a DNC function of minimal degree. 

\begin{theorem}[Khan]\label{thm:minimal-X-DNC}
	Given any oracle $X$, there is a function that is $\DNC$ relative to $X$ and of minimal degree.
\end{theorem} 
	
	Kumabe and Lewis \cite{KumabeLewis} provided a simplified version of Kumabe's original arguments \cite{Kumabe}. Our proof reuses much of the combinatorial machinery developed in their proof, but differs in several key aspects. Kumabe and Lewis use partial trees with computable domains, hence the function they produce is hyperimmune-free. We use partial trees with noncomputable domains, out of necessity: by Theorem~\ref{thm:ZPDNRhyperimmune}, any $\DNC$ function relative to $0'$ is hyperimmune. Further, it suffices in the Kumabe-Lewis construction to work with bad sets of constant bushiness. This is not the case here; our bad sets are $h$-small for some order function $h$. In our approach to bad sets of varying bushiness, we use ideas from Cai and Greenberg's result in \cite{CaiThesis} that there exist degrees $\deg{a}$ and $\deg{b}$ such that $\deg{a}$ is minimal and $\DNC$ and $\deg{b}$ is $\DNC$ relative to $\deg{a}$ and a strong minimal cover of $\deg{a}$.

\subsection{Definitions and notation}
\begin{definition}
	Let $h$ be an order function. Given $\sigma \in \bstrings$, we say that a set $B \subseteq \bstrings$ is \emph{$h$-big above $\sigma$} if there is a finite $h$-bushy tree $T$ above $\sigma$ such that all its leaves are in $B$. If $B$ is not $h$-big above $\sigma$ then we say that $B$ is \emph{$h$-small} above $\sigma$.
\end{definition}

It is easy to see that the smallness preservation property, concatenation property and small set closure property all continue to hold when one replaces the constants governing bushiness with order functions.

For an order function $g$ and $l \in \omega$, let $w_g(l)$ denote $\prod_{i < l} g(i)$ and let $r(g, l)$ denote $2^{3 + 3w_g(l)}$.

In order to simplify our calculations, throughout this proof we restrict ourselves to order functions that only take values that are powers of two.

\begin{definition}
	Suppose $h(n) = 2^{h'(n)}$ and $g(n) = 2^{g'(n)}$ are order functions, where $h', g' : \omega \rightarrow \omega$. The \emph{middle} of $h$ and $g$ is the order function $\+M(h, g)$ defined by \[\+M(h, g)(n) = 2^{\floor{\frac{h'(n) + g'(n)}{2}}}.\]
\end{definition}

\begin{definition}
	Suppose $h$ and $g$ are order functions. We say the pair $(h, g)$ \emph{allows splitting above $N \in \omega$} if 
	\begin{enumerate}
		\item $h(N) \ge g(N)$,
		\item for $n \ge N$, $h(n)/g(n)$ is nondecreasing, and
		\item there is an increasing sequence $\pair{l_i}_{i \in \omega}$ of natural numbers with $l_0 \ge N$ such that $h(l_i)/g(l_i) \ge (r(h, l_i))^i$.
	\end{enumerate}
	We say $(h, g)$ allows splitting if it allows splitting above some $N \in \omega$. We call the sequence $\pair{l_i}$ the \emph{splitting levels} for $(h, g)$.
\end{definition}

\begin{lemma}
	Let $h$ and $g$ be order functions such that $(h, g)$ allows splitting. Let $m = \+M(h, g)$. Then $(m, g)$ and $(h, m)$ allow splitting.
\end{lemma}
\begin{proof}
	We provide the argument for $(m, g)$. Suppose $(h, g)$ allows splitting above $N$ and $\pair{l_i}$ is the sequence of splitting levels for $(h, g)$. Note that conditions (1) and (2) in the definition above are satisfied by $(m, g)$ above $N$.
	
	We verify condition (3). Suppose that $h(n) = 2^{h'(n)}$ and $g(n) = 2^{g'(n)}$. Note first that for each $n \ge N$, 
	\[\frac{m(n)}{g(n)} = 2^{\floor{\frac{h'(n) + g'(n)}{2}}-g'(n)} = 2^{\floor{\frac{h'(n) - g'(n)}{2}}} \ge \frac{2^{\frac{h'(n) - g'(n)}{2}}}{2}.\]
	
	It follows that for each $i \in \omega$, \[\frac{m(l_{2i + 2})}{g(l_{2i + 2})} \ge \frac{2^{\frac{h'(l_{2i + 2}) - g'(l_{2i + 2})}{2}}}{2} \ge \frac{(r(h, l_{2i + 2}))^{\frac{2i + 2}{2}}}{2} = \frac{(r(h, l_{2i + 2}))^{i + 1}}{2} \ge (r(m, l_{2i + 2}))^i,\] so $\pair{l_{2i + 2}}_{i \in \omega}$ is a sequence of splitting levels for $(m, g)$.
	
	A similar calculation shows that $(h, m)$ allows splitting.
\end{proof}
	It is not hard to verify that if $(h, g)$ allows splitting then for any $c \in \omega$, so do $(h, 2^c g)$ and $(\max(h/2^c, 2), g)$.

\subsection{The partial order}

The forcing conditions are of the form $(\sigma, T, B, h_T, h_B)$, where
\begin{itemize}
	\item the tree $T$ is partial computable (some nodes may be terminal) and exactly $h_T$-bushy above $\sigma$,
	\item $B$ includes the terminal nodes in $T$, is upward closed and is $h_B$-small above $\sigma$, and
	\item $(h_T, h_B)$ allows splitting above $|\sigma|$.
\end{itemize}

Only $\sigma$, $T$ and $B$ contribute to the ordering. Let $h_M$ denote $\+M(h_T, h_B)$. By extending $\sigma$ appropriately, we can assume that $h_M(n)/16 \ge h_B(n)$ for all $n \ge |\sigma|$. 

Note that we have no access to the set $B$ (it is not c.e.). Since the terminal nodes of $T$ are contained in the bad set $B$, the conditions that force $f_\+G$ to be total are dense in this partial order.

As before, we can assume that the bad set is $h_B$-closed. In other words, if $\tau$ is any string in $T \setminus B$ then $B$ is $h_B$-small above $\tau$.

\subsection{Forcing $\Gamma^{f_\+G}$ to be partial}

Let $C_n = \dset{\tau \in T}{\Gamma^\tau(n) \converges}$. Given a condition $(\sigma, T, B, h_T, h_B)$ and a functional $\Gamma$ we say we can \emph{force $\Gamma^{f_\+G}$ to be partial} if there is a $\tau$ on $T$ extending $\sigma$ and an $n$ such that the set $C_n \cup B$ is $h_M$-small above $\tau$.
If this is the case, then we let $T'$ be the full subtree of $T$ above $\tau$. The condition $(\tau, T', C_n \cup B, h_T, h_M)$ extends $(\sigma, T, B, h_T, h_B)$, while forcing $\Gamma^{f_\+G}(n)\diverges$. From now on we assume that we cannot force $\Gamma^{f_\+G}$ to be partial. It follows that for every $n$, and every $\tau \in T \setminus B$, $C_n \setminus B$ is $h_M/2$-big above $\tau$. Applying this fact iteratively we obtain the following claim:

\begin{claim}\label{clm:convergence}
	For any $\tau \in T \setminus B$ extending $\sigma$ and any $n$, there is an $A \subset T \setminus B$, $h_M/2$-big above $\tau$, such that for every $\rho \in A$, $\Gamma^\rho \res n$ is defined.
\end{claim}

\subsection{Forcing $\Gamma^{f_\+G}$ to be computable}
	It is worth pointing out here how our argument for this case of the forcing differs from the one in Kumabe-Lewis. As we have mentioned, the bad sets in their argument are c.e., and they make strong use of this fact in an \emph{effective} simultaneous construction of a refined subtree and a real $Y$ that it is the image of $\Gamma$ on every path on this subtree (and hence computable). We do not have access to the bad set, since we ultimately want it to include the set of strings that are non-DNC relative to $X$. So we construct a sufficiently bushy subtree \emph{noneffectively}, and let $Y$ be the image under $\Gamma$ of this tree. Under the assumptions we make in this case of the forcing, $Y$ turns out to be computable. 

\begin{definition}
	Let $g$ be an order function. A \emph{$g$-big splitting above $\tau \in T$} is a pair of sets $A_0 \subset T$ and $A_1 \subset T$, both $g$-big above $\tau$, such that for any $\tau_0 \in A_0$ and $\tau_1 \in A_1$, $\Gamma^{\tau_0} \mid \Gamma^{\tau_1}$. We say that $A_0$ and $A_1$ are \emph{$\Gamma$-splitting}.
\end{definition}

Suppose that there is a $\tau \in T \setminus B$ extending $\sigma$ such that we cannot find any $h_M/16$-big splitting above $\tau$. Under this assumption, we construct a real $Y$ with the property that for each $n \in \omega$, the set of $\rho$ on $T$ such that $\Gamma^\rho \res n = Y \res n$ is $h_M/4$-big above $\tau$. It follows immediately that $Y$ is computable. To compute it up to $n$ bits, we search for an $h_M/4$-bushy tree $A \subset T$ above $\tau$ every leaf of which gives the same $n$ bits of convergence via $\Gamma$. These bits must agree with $Y$, otherwise we will have obtained an $h_M/16$-big splitting above $\tau$. Further, if we let $D = \dset{\rho \in T}{\Gamma^\rho \mid Y}$, then $D$ is $h_M/16$-small above $\tau$. It follows that $B \cup D$ is $h_M$-small above $\tau$, so letting $T'$ be the full tree above $\tau$, the condition $(\tau, T', B \cup D, h_T, h_M)$ extends $(\sigma, T, B, h_T, h_B)$ while forcing $\Gamma^{f_\+G}$ to be computable.

We construct $Y$ bit by bit, letting $Y_0 = \Gamma^\tau$. We also assume inductively that there is a set $S_i \subset T \setminus B$ that is $h_M/4$-big above $\tau$ and for every $\rho \in S_i$, $\Gamma^\rho \upto i + |Y_0| = Y_i$. Let $S_0$ consist of just $\tau$.

Given $Y_i$ and $S_i$, we proceed as follows. Above each leaf $\rho$ of $S_i$, there is an $h_M/2$-big set of strings $A_\rho$ such that for each $\nu \in A_\rho$, $\Gamma^\nu(|Y_i|)$ is defined. $A_\rho$ can then be thinned out to a set $A'_\rho$ that is $h_M/4$-big above $\rho$ and such that for each $\nu \in A'_\rho$, $\Gamma^\nu(|Y_i|)$ converges to the same value $c_\rho$. Next, since $S_i$ is $h_M/4$-big above $\tau$, there is a $V \subset S_i$, $h_M/8$-big above $\tau$, such that for each $\rho \in V$, $c_\rho$ is the same value, say $j$. Let $Y_{i+1} = Y_i j$. Note that $V' = \cup \dset{A'_\rho}{\rho \in V}$ is $h_M/8$-big above $\tau$ and for each $\nu \in V'$, $\Gamma^{\nu} \succeq Y_{i+1}$. Let $S_{i+1} = \dset{\nu \in C}{\Gamma^\nu \succeq Y_{i+1}}$. The set $C \setminus S_{i+1}$ must be $h_M/16$-small above $\tau$, otherwise $C \setminus S_{i+1}$ and $V'$ form an $h_M/16$-big splitting above $\tau$. It follows that $S_{i+1}$ is $h_M/4$-big above $\tau$.
 
\subsection{Forcing $\Gamma^{f_\+G} \ge_T f_\+G$}

We work now under the additional assumption that for each $\tau \in T \setminus B$ extending $\sigma$ there is a $h_M/16$-big splitting above $\tau$.

We refine $T$ to a subtree $S$ that has the \emph{delayed splitting property}: above each $\tau \in S \setminus B$, there are levels $l' > l > |\tau|$ such that if $\rho_0$ and $\rho_1$ are any two extensions of $\tau$ on $S$ of length $l$, and $\rho'_0 \succ \rho_0$ and $\rho'_1 \succ \rho_1$ are extensions on $S$ of length $l'$, then $\Gamma^{\rho'_0} \mid \Gamma^{\rho'_1}$. 

The statement of the following lemma has been slightly modified from the original in order to apply to trees of varying bushiness:

\begin{lemma}[Kumabe, Lewis \cite{KumabeLewis}] \label{lem:splitting}
	Let $\Gamma$ be a functional. Let $A$ be $4g$-big above $\alpha$ and $B$ be $4h$-big above $\beta$, where $g$ and $h$ are order functions. Suppose that above every $\tau \in A$, there exist $\Delta_{\tau, 0}$ and $\Delta_{\tau, 1}$, such that they are both $4g$-big above $\tau$ and are $\Gamma$-splitting. Let $A' = \cup_{\tau, i} \Delta_{\tau, i}$ and let $v = \max \dset{|\Gamma^\rho|}{\rho \in A'}$. If for every leaf $\sigma$ of $B$, $|\Gamma^\sigma| > v$, then there is an $A'' \subseteq A'$ and a $B' \subseteq B$, $g$-big above $\alpha$ and $h$-big above $\beta$ respectively, that are $\Gamma$-splitting. 	
\end{lemma}
\begin{proof}
	Let $\sigma_0 = \estr$ and $B_0 = B$.
	
	Assume inductively that we have $\sigma_s$ of length $s$ and $B_s$, $h$-big above $\beta$, such that for all $\rho \in B_s$, $\Gamma^\rho \succeq \sigma_s$.
	
	If $\dset{\tau \in A'}{\Gamma^\tau \mid \sigma_s}$ is $g$-big above $\alpha$ then we are done. If not, then either 
	\begin{enumerate}
		\item $A_1 = \dset{\tau \in A'}{\Gamma^\tau \preceq \sigma_s}$ is $g$-big above $\alpha$ or
		\item $A_2 = \dset{\tau \in A'}{\Gamma^\tau \text{ properly extends } \sigma_s}$ is $g$-big above $\alpha$.
	\end{enumerate}
	If (1) holds then let $V$ be the set of leaves of $A$ that have an extension in $A_1$. For each $\tau \in V$, the set of strings in $A_1$ extending $\tau$ must lie entirely in one of the $\Delta_{\tau, i}$. Let $\Delta'_\tau$ denote the other member of the splitting above $\tau$. Then $\cup\dset{\Delta'_\tau}{\tau \in V}$ is $g$-big above $\alpha$ and splits with $B_s$.
	
	Next, assume (2) holds, which implies that $|\sigma_s| < v$. If $\dset{\tau \in B}{ \Gamma^\tau \mid \sigma_s}$ is $h$-big above $\beta$, then we are done. If not, then it must be the case that $D = \dset{\tau \in B}{\Gamma^\tau \succeq \sigma_s}$ is $2h$-big above $\beta$. $D$ can be partitioned into the sets $D_i = \dset{\tau \in D}{\Gamma^\tau(|\sigma_s|) = i}$, one of which must be $h$-big above $\beta$, say $D_j$. Let $B_{s+1} = D_j$ and let $\sigma_{s+1} = \sigma_s j$ and continue the construction. Since this process cannot continue indefinitely, we will obtain the required splitting via one of the other alternatives.
\end{proof}

\begin{claim} \label{clm:splitting-ratio}
	Suppose $\tau_0, ..., \tau_k$ are nodes of length $l$ in $T \setminus B$, $k < w_{h_M}(l)$ and that $h_M(l)/h_B(l) \ge r(h_M, l)$. Then there is a sequence of sets $A_0, ..., A_k$, where $A_j$ is $(h_M/2^{3+3k})$-big above $\tau_j$ and which are pairwise $\Gamma$-splitting.
\end{claim}
\begin{proof}
	The proof is by induction on $k$. Suppose we already have $A_0, ..., A_k$, where each $A_j$ is $(h_M/2^{3+3k})$-big above $\tau_j$ and the collection is pairwise $\Gamma$-splitting. Let $\tau_{k+1}$ be an additional node of length $l$ that is not in $B$ and let $q = h_M/2^3$. 
	
	Note that since $w_{h_M}(l) > k + 1$, $h_B(l) < q(l)/2^{3k+1}$. So we first refine each $A_j$ to a $\Pi_j$ where $\Pi_j$ is $(q/2^{3k+1})$-big above $\tau_j$ and $\Pi_j \cap B = \emptyset$. If $\rho$ is a leaf of $\Pi_j$, then it is not in $B$ and since $q/2^{3k+1} \le h_M/16$, we can find a $q/2^{3k+1}$-bushy splitting, say $D_{\rho, 0}$ and $D_{\rho, 1}$, above $\rho$. We let $\Pi'_j = \cup_{i, \rho} D_{\rho, i}$.
	
	Let $m$ be the longest length of the image of $\Gamma$ on any string in any of the $\Pi'_j$. Appealing to Claim \ref{clm:convergence}, we let $\Delta_0$ be a $q$-big set above $\tau_{k+1}$ such that each leaf of $\Delta_0$ gives at least $m+1$ bits of convergence via $\Gamma$. We now apply Lemma \ref{lem:splitting} on $\Pi'_0$ and $\Delta_0$, obtaining $A'_0 \subset \Pi'_0$ and $\Delta_1 \subset \Delta_0$, which are $\Gamma$-splitting and where the former is $q/2^{3(k+1)}$-big above $\tau_0$ and the latter is $q/4$-big above $\tau_{k+1}$. Next, we apply Lemma \ref{lem:splitting} to the pair $\Pi'_1$ and $\Delta_1$, obtaining $A'_1 \subset \Pi'_1$ and $\Delta_2 \subset \Delta_1$, which are $\Gamma$-splitting and where $A'_1$ is $q/2^{3(k+1)}$-big above $\tau_1$ and $\Delta_2$ is $q/4^2$-big above $\tau_{k+1}$. After $k+1$ applications of Lemma \ref{lem:splitting}, we will have obtained $A'_0$ through $A'_k$ and $\Delta_{k+1}$, which are pairwise $\Gamma$-splitting. Moreover, $\Delta_{k+1}$ is $q/2^{2(k+1)}$-big above $\tau_{k+1}$, so we can let $A'_{k+1} = \Delta_{k+1}$. 
\end{proof}
	
	Our argument here differs once again in a crucial way from Kumabe and Lewis's. Suppose we have defined the delayed splitting tree $S$ up to a certain level and let $\tau$ be one of the leaves of this finite tree. In order to continue the construction above $\tau$, we must find a sufficiently bushy splitting above $\tau$. In the Kumabe-Lewis argument, such a splitting will be found, or $\tau$ will be seen to enter the bad set. In either case, the construction of the tree $S$ is in no danger of ``stalling''. Here, however, we have no access to the bad set, so we may end up searching in vain for a splitting. In order to get around this, we will only ask for splittings above sufficiently bushy many leaves of the current approximation to $S$, a situation that we can guarantee, and add the remaining leaves to the bad set. Thus, we will be adding lots of strings to the bad set at each level of the construction. The following lemma is critical to preserving its smallness when we do so:
	\begin{lemma}\label{lem:smallness-preservation}
		Let $g$ be an order function. Suppose $A \subset \bstrings$ is $g$-small above $\sigma \in \bstrings$, and suppose $\tau \in \bstrings$ extends $\sigma$ and $A$ contains no extension of $\tau$. If $B$ is a set of strings extending $\tau$ that is $g$-small above $\tau$, then $A \cup B$ is $g$-small above $\sigma$.
	\end{lemma}
	\begin{proof}
		Suppose otherwise, i.e., there is a $g$-bushy tree $T$ above $\sigma$ with leaves in $A \cup B$. It must be the case that some leaves of $T$ are in $B$. Since every string in $B$ extends $\tau$, $\tau \in T$. This means that there is a tree $T'$ that is $g$-bushy above $\tau$ whose leaves are in $B$, namely, the tree consisting of all strings in $T$ that are comparable with $\tau$. This is a contradiction.
	\end{proof}

	Let $\pair{l_i}$ be the sequence of splitting levels for the pair $(h_M, h_B)$. We begin by defining $h_S$. Let $j_i = l_{i+1}$. For $n < j_0$, let $h_S(n) = h_M(n)$. For $j_{i+1} > n \ge j_i$, let $h_S(n) = h_M(j_i) / r(h_S, j_i)$. Then for each $i$, 
	\[ \frac{h_S(j_i)}{h_B(j_i)} = \frac{h_M(j_i)}{h_B(j_i) r(h_S, j_i)} \ge \frac{(r(h_M, j_i))^{i+1}}{r(h_S, j_i)} \ge (r(h_S, j_i))^i.\]
	Hence the pair $(h_S, h_B)$ allows splitting above $|\sigma|$.

	We now describe how we build the partial computable tree $S$. We start by letting $S_0$ be an $h_S$-bushy subtree of $T$ above $\sigma$ with leaves of length $l_1$ or less such that if $D_0$ is the set of leaves of $S_0$ of length strictly smaller than $l_1$, then $D_0$ is $h_B$-small above $\sigma$. Since the terminal nodes of $T$ are contained in $B$, such a tree must exist. We declare the nodes in $D_0$ terminal and the leaves of $S_0$ that are of length $l_1$ to be the \emph{children of $\sigma$}. Throughout the construction we will maintain the property that if $\tau \in S$ has children in $S$, then they are all of the same length and that length is a splitting level for the pair $(h_M, h_B)$.
	
	 At a stage $s$ of the construction, we will have built a finite approximation $S_s$ of $S$, and accumulated a set $D_s$ of nodes on $S_s$ that we have declared terminal. $D_s$ will always be $h_B$-small above $\sigma$.
	
	 Suppose that $\tau \in S_s$ has a set $C_\tau$ of children of length $l_i$ and that they are leaves of $S_s$. If we have not already done so, we initiate a search for a subset $C'_\tau$ of $C_\tau$ such that $C_\tau \setminus C'_\tau$ is $h_B$-small above $\tau$, and for each $\rho \in C'_\tau$, there is a $A_\rho$, $h_S$-bushy above $\rho$ such that the collection $\dset{A_\rho}{\rho \in C'_\tau}$ is pairwise $\Gamma$-splitting.
	
	If $\tau \notin B$ then this search must terminate. To see why this is the case note first that $B$ is $h_B$-small above $\tau$. Let $\rho_0, ..., \rho_k$ be the strings in $C_\tau \setminus B$. Since $l_i$ is a splitting level for $(h_M, h_B)$, $h_M(l_i)/h_B(l_i) \ge r(h_M, l_i)$. Moreover, $w_{h_M}(l_i) \ge w_{h_S}(l_i) > k$. By Claim \ref{clm:splitting-ratio}, there are $A_0, ..., A_k$, with $A_j$  $h_M/2^{3 + 3k}$-big above $\rho_j$, that are pairwise $\Gamma$-splitting. Now \[\frac{h_M(n)}{2^{3 + 3k}} \ge \frac{h_M(n)}{2^{3 + 3w_{h_S}(l_i)}} = \frac{h_M(n)}{r(h_S, l_i)} \ge h_S(n) \] for $n \ge l_i$, so we can refine the $A_j$ to subtrees that are $h_S$-bushy.
	
	If $C'_\tau$ is found, then we extend each $\rho \in C'_\tau$ by $A_\rho$. Note that by Lemma~\ref{lem:smallness-preservation}, $D_s \cup (C_\tau \setminus C'_\tau)$ is $h_B$-small above $\sigma$, since $D_s$ initially contains no extension of $\tau$ and $C_\tau \setminus C'_\tau$ is $h_B$-small above $\tau$. So we can add $C_\tau \setminus C'_\tau$ to $D_s$.
	
	 Next, for each $\rho \in C'_\tau$ we wish to extend the leaves of $A_\rho$ $h_S$-bushily to the next splitting level for $(h_M, h_B)$. Let $L_\rho$ be the set of leaves of $A_\rho$, and let $m = \max \dset{|\nu|}{\nu \in L_\rho}$. Let $l$ be least splitting level for $(h_M, h_B)$ greater equal to $m$. We begin a search for an $L'_\rho \subseteq L_\rho$ such that $L_\rho \setminus L'_\rho$ is $h_B$-small above $\rho$ and above each $\nu \in L'_\rho$ there is an $h_S$-bushy tree with leaves of length $l$. Note that if $\rho \notin B$, this search must terminate. When we find such an $L'_\rho$, we extend all its elements $h_S$-bushily to level $l$, declaring the new leaves to be the children of $\rho$ and add $L_\rho \setminus L'_\rho$ to $D_s$. The same argument as before shows that $D_s$ remains $h_B$-small above $\sigma$.
	
	The resulting tree $S$ is $h_S$-bushy and if we let $D = \cup_s D_s$, then the new bad set $D \cup B$ is $2 h_B$-small above $\sigma$. It is clear that the construction halts above a node $\tau \in S$ if it is either in $B$ or we have declared it to be terminal by adding it to $D$, and so $B \cup D$ contains all the terminal nodes of $S$. By extending $\sigma$, we can ensure that $(h_S, 2 h_B)$ allows splitting above $|\sigma|$. For such a $\sigma$, the condition $(\sigma, S, D \cup B, h_S, 2 h_B)$ extends $(\sigma, T, B, h_T, h_B)$ and forces $\Gamma^{f_\+G} \ge_T f_\+G$.

	This completes the proof of the theorem.

\subsection{A question} The $\DNC$ function yielded by Theorem~\ref{thm:minimal-X-DNC} is computably bounded, but we do not know if the construction can be carried out below a given order function. This difficulty is also present in Kumabe and Lewis's argument, and arises from the fact that applying the splitting method of Lemma~\ref{lem:splitting} repeatedly for each pair of leaves above which we need to find a splitting, as we do in Claim~\ref{clm:splitting-ratio}, is rather costly in terms of bushiness and this cost grows exponentially in the number of leaves. Thus it is not simply a matter of delaying the task of splitting until we have reached a level where the bushiness of the surrounding tree is sufficiently high: By extending the leaves of our subtree bushily, we generate exponentially \emph{more} leaves, and the bushiness required of the surrounding tree for splitting above these leaves grows at a rate that is at least doubly exponential.

Whether this difficulty can be surmounted is a question of considerable interest. If we were able to carry out the construction below the order function given by Theorem~\ref{thm:slowhdim1}, then by its relativized version, we could conclude that for every oracle $X$, there is a real of effective Hausdorff dimension 1 relative to $X$ that is of minimal degree. This would imply that the classical Hausdorff dimension of the set of reals of minimal degree, which is as yet unknown, is 1.

\begin{question}
	For every oracle $X$, and for every order function $h$, is there an $h$-bounded function that is $\DNC$ relative to $X$ and of minimal degree?
\end{question}

\section{Appendix}

\begin{theorem}[Miller]\label{thm:ZPDNRhyperimmune}
	Every function that is $\DNC$ relative to $0'$ is of hyperimmune degree.
\end{theorem}
\begin{proof}
	The argument is a modification of the proof by D. A. Martin of the fact that almost every real is of hyperimmune degree, as presented in Downey and Hirschfeldt \cite{DowneyHirschfeldtBook}. We construct a reduction $\Psi$ on $\omega^\omega$ such that whenever $\Psi^g$ is total for a $g \in \omega^\omega$, $\Psi^g$ is not computably dominated. Moreover, any function that is $\DNC$ relative to $0'$ can compute such a $g$.

\subsection*{Construction of $\Psi$.}
	 The construction is comprised of subconstructions that operate simultaneously and independently above each string. If $\tau \in \bstrings$ is of length $n$, then the subconstruction above $\tau$ attempts to do the following for each extension $\tau i$ of $\tau$ \emph{in succession}: 
	\begin{enumerate}
		\item Reserve a number $m$ for $\tau i$ such that $\Psi^{\tau i}(m)$ has not already been defined and $m$ has not already been reserved for any initial segment or extension of $\tau i$.
		\item If $\varphi_n(m)$ converges, set $\Psi^{\tau i}(m) = \varphi_n(m) + 1$ and move on to $\tau (i + 1)$. If $\varphi_n(m)$ never converges, then we say that the subconstruction above $\tau$ \emph{halts at $\tau i$}.
	\end{enumerate} 

	In addition, at each stage $s$ of the construction, let $T_s$ be the (finite) set of strings for which $s$ has been reserved. For each $f \in \omega^\omega \setminus [T_s]^\prec$, set $\Psi^f(s) = 0$. 

\subsection*{Verification.} 
 	Note that if $g$ is such that no subconstruction halts at one of its initial segments, then $\Psi^g$ is total: If $m$ was reserved for an initial segment of $g$, then $\Psi^g(m)$ converges. Otherwise, at stage $m$ we would have set $\Psi^g(m) = 0$.
 	
 	\begin{claim}
		If $\Psi^g$ is total, then it is not computably dominated.
	\end{claim}
 	\begin{proof}
 		If $\tau \prec \tau i \prec g$ where $\tau$ is of length $n$, then the subconstruction above $\tau$ did not halt at $\tau i$. If it halted at $\tau j$ for some $j < i$, then $\varphi_n$ never converged on the number reserved for $\tau j$. If it proceeded beyond $\tau i$, then $\Psi^g(m) = \varphi_n(m) + 1$, where $m$ is the number reserved for $\tau i$. In either case, $\varphi_n$ does not dominate $\Psi^g$.
 	\end{proof}
 	
 	\begin{claim}
 		If $f$ is $\DNC$ relative to $0'$ then it computes a function $g$ such that no subconstruction halts at an initial segment of $g$, and therefore, $\Psi^g$ is total.
 	\end{claim} 
 	\begin{proof}
 		We construct $g$ by initial segments $\pair{\tau_s}_{s \in \omega}$. Given $\tau_s$, let $n$ be such that $\varphi^{0'}_n(n)$ converges to $i$ if and only if the subconstruction above $\tau_s$ halts at $\tau_s j$. Then let $\tau_{s + 1} = \tau_s f(n)$. Finally, let $g = \bigcup_s \tau_s$.
 	\end{proof}

 	This concludes the proof of the theorem.
\end{proof}
\bibliography{research}{}
\bibliographystyle{plain}

\end{document}